\documentclass[12pt]{amsart}
\usepackage{amsthm,amsmath,amsfonts,amssymb,latexsym,amscd,mathrsfs,hyperref,cleveref}
\usepackage{graphicx}
\usepackage{amsaddr}
\usepackage{cite}

\theoremstyle{plain}
\newtheorem{thm}{Theorem}
\newtheorem{lem}{Lemma}

\theoremstyle{definition}

\newtheorem*{note}{Notation}

\crefname{thm}{theorem}{theorems}
\crefname{lem}{lemma}{lemmas}

\newcommand{\mb}{\mathbb}
\newcommand{\mc}{\mathcal}

\renewcommand{\mod}{\operatorname{mod}}

\def \a{\alpha} \def \b{\beta} \def \d{\delta} \def \e{\varepsilon} \def \g{\gamma}  \def \l{\lambda}  \def \t{\theta} 

\numberwithin{equation}{section}

\renewcommand{\labelenumi}{\setlength{\labelwidth}{\leftmargin}
   \addtolength{\labelwidth}{-\labelsep}
   \hbox to \labelwidth{\theenumi.\hfill}}

\begin{document}
\title{Primes in arithmetic progressions to spaced moduli. III}
\author{Roger Baker}
\address{Department of Mathematics\\
Brigham Young University\\
Provo, UT 84602, U.S.A\\
baker@math.byu.edu}

 \begin{abstract}
Let
 \[E(x,q) = \max_{(a,q)\, =\, 1} \Bigg|
 \sum_{\substack{n\, \le\, x\\
 n\, \equiv\, a\, (\mod q)}} \Lambda(n)
 - \frac x{\phi(q)}\Bigg|.\]
We show that, for $S$ the set of squares,
 \[\sum_{\substack{q\, \in\, S\\
 Q\, <\, q\, \le\, 2Q}} E(x, q) \ll_{A,\e} x\, 
 Q^{-1/2}(\log x)^{-A} \]
for $\e > 0$, $A > 0$, and $Q \le x^{1/2-\e}$. This improves a theorem of the author.
 \end{abstract}
 

\maketitle

\section{Introduction} Let
 \[
 E(x,q) = \max_{(a,q)=1} \Bigg|\sum_{\substack{
 n \le x\\
 n\, \equiv\, a (\mod q)}} \Lambda(n) - 
 \frac x{\phi(q)}\Bigg|,
 \]
where $\Lambda$ is the von Mangoldt function. Let
 \[S_f = \{f(k) : k \in \mb N\},\]
where $f$ is a polynomial of degree $d \ge 2$ with integer coefficients and positive leading coefficient. In analogy with the Bombieri-Vinogradov theorem, we would like to show that
 \begin{equation}\label{eq1.1}
\sum_{\substack{q\, \in\, S_f\\
 Q\, <\, q\, \le\, 2Q}} E(x, q) \ll_{A,\e} x\, Q^{1/d-1}(\log x)^{-A} 
 \end{equation}
for $\e > 0$, $A > 0$ and $Q \le x^{1/2-\e}$. In the general case, \eqref{eq1.1} is known only for $Q \le x^{9/20-\e}$, and in the special case $f(X) = X^2$, for $Q \le x^{43/90 - \e}$ \cite{rcb}.

Here we refine the approach in \cite{rcb} for $f(X) = X^2$.

 \begin{thm}\label{thm1}
Let $f(X) = X^2$. Then \eqref{eq1.1} holds for $Q \le x^{1/2-\e}$.
 \end{thm}

To prove \Cref{thm1}, we sharpen the auxiliary results on pp.~147--150 of \cite{rcb}. With a little modification, we are then able to complete the proof of \Cref{thm1} by arguing as in \cite{rcb}. The key new result is \Cref{lem2} below, which strengthens Lemma 11 of \cite{rcb}.  Thanks are due to James Maynard for suggesting in conversation the line of argument used to prove \Cref{lem2}.

 \begin{note}
We write
 \[\|\t\| = \min_{n\in \mb Z} |\t - n|\]
and, for complex numbers $c_1, \ldots, c_N$,
 \[\|c\|_2 = \left(\sum_{n=1}^N |c_n|^2\right)^{1/2}.\]
The $k$-th Riesz mean is defined by
 \begin{equation}\label{eq1.2}
A_k(x, q, a, d) = \frac 1{k!} \sum_{\substack{
\ell\, \le\, x\\
\ell\, \equiv\, a(\mod q)\\
\ell\, \equiv\, 0(\mod d)}} \left(\log\, \frac x\ell\right)^k
\quad (k=0, 1, \ldots)
 \end{equation}
and we write
 \begin{equation}\label{eq1.3}
r_k(x, q, a, d) = A_k(x, q, a, d) - \frac x{qd}. 
 \end{equation}
It is convenient to write $a^{(q)}$ for an arbitrary integer with $(a^{(q)}, q) = 1$.

We suppose, as we may, that $x$ is large and $\e$ is sufficiently small, and write $\d = \e^2$. Except in \Cref{lem5}, implied constants depend at most on $\e$ or, when $A$ appears in the result, on $\e$ and $A$.
 \end{note}

The conductor of a primitive Dirichlet character $\chi$ is denoted by $C(\chi)$.

\section{The large sieve for square moduli}\label{sec2}

 \begin{lem}\label{lem1}
Let $\Delta > 0$ and $Q \ge 1$. For $\b$ real, let $\mc N(\b)$ denote the number of relatively prime pairs $a$, $q$, $1 \le a \le q^2$, $q \le Q$, with
 \[\left\|\frac a{q^2} - \b\right\| \le \Delta.\]
Then
 \[\mc N(\b) \ll (Q\Delta^{-1})^\e(Q^3\Delta + Q^{1/2}).\]
 \end{lem}
 
 \begin{proof}
This is due to Baier and Zhao \cite[Section 11]{bz}.
 \end{proof}
 
 \begin{lem}\label{lem2}
Let $Q \ge 1$. Let $a_1, \ldots, a_N$ be complex numbers,
 \[T(\a) = \sum_{n=1}^N a_n e(n\a).\]
Let $g \in \mb N$. Then
 \begin{equation}\label{eq2.1}
\sum_{q\le Q} \sum_{\substack{a = 1\\
(a, gq^2) = 1}}^{gq^2} \left|T\left(\frac q{gq^2}\right)\right|^2 \ll (QN)^\e \left(1 + \frac gN\right)(gQ^3 + Q^{1/2}N) \|a\|_2^2. 
 \end{equation}
 \end{lem}
 
 \begin{proof}
We first show that, for real $\a$ and $\Delta > 0$, the number $\mc M(\a)$ of solutions of
 \begin{equation}\label{eq2.2}
\left\|\frac a{gq^2} - \a\right\| \le \Delta, 0 \le a \le gq^2 -1, (a, gq^2) = 1, 1 \le q \le Q,
 \end{equation}
satisfies
 \begin{equation}\label{eq2.3}
\mc M(\a) \ll (1 + g\Delta)(Q^3(g\Delta) + Q^{1/2})(Q\Delta^{-1})^\e.
 \end{equation}

To see this, write $a = b + q^2n$, $0 \le n < g$, $0 \le b < q^2$. Then \eqref{eq2.2} implies
 \[\left\|\frac b{q^2} - g\a\right\| = \left\|
 \frac{b + q^2n}{q^2} - g\a\right\| \le g\Delta.\]
The number of possible $b$ is 
 \[\ll (Q^3(g\Delta) + Q^{1/2})(Q\Delta^{-1})^\e\]
by \Cref{lem1}. Once $b$ is fixed, \eqref{eq2.2} implies
 \[\left\|\frac b{gq^2} + \frac nq - \a\right\| \le \Delta.\]
There are at most $2g\Delta + 1$ possible $n$, and the bound \eqref{eq2.3} follows.

By \cite[Theorem 2.1]{mont}, the left-hand side of \eqref{eq2.1} is bounded by
 \begin{equation}\label{eq2.4}
\ll (N + \Delta^{-1})\left(\max_{\a\, \in\, \mb R} \mc M(\a)\right) \|a\|_2^2,
 \end{equation}
for any $\Delta > 0$. We take $\Delta = N^{-1}$ and apply \eqref{eq2.3} to obtain the lemma. 
 \end{proof}
 
 \begin{lem}\label{lem3}
Let $Q = x^\t$ and $0 < \l \le \t$, $x \ge M \gg x^{\t - \l}$. Let $c_1, \ldots, c_M$ be complex numbers. Let
 \[T(\l) = \sum_{Q < q^2 \le 2Q}\
 \sum_{\substack{\chi\ (\mod{q^2})\\
 x^\l\, <\, C(\chi)\, \le\, 2x^\l}} \Bigg|\sum_{m=1}^M
 c_m \chi(m)\Bigg|^2.\] 
Then
 \[T(\l) \ll x^\e(Q^{1/2} x^\l + Q^{3/4} Mx^{-\l/2})
 \|c\|_2^2.\]
 \end{lem}
 
 \begin{proof}
For a character $\chi\pmod{q^2}$ counted in $T(\l)$, induced by a primitive character $\chi'\pmod{C(\l)}$, we have
 \begin{equation}\label{eq2.5}
C(\chi) = gk^2 \in (x^\l, 2x^\l]
 \end{equation}
with $g$ squarefree, $k \in \mb N$; and
 \[\chi(m)= \begin{cases}
 \chi'(m) & \text{if } (m, q) = 1\\
 0 & \text{if } (m, q) > 1.
 \end{cases}\]
Since $C(\chi)\mid q^2$, we have
 \[vgk^2 = q^2 \in (x^\t, 2x^\t]\]
for a natural number $v$. Obviously $v = gt^2$, $t \in \mb N$,
 \begin{equation}\label{eq2.6}
x^{\t/2} < q = gtk \le (2x^\t)^{1/2}.
 \end{equation}
It follows that
 \begin{equation}\label{eq2.7}
T(\l) \le \sum_{\substack{g\, \ge\, 1,\, t \, \ge 1\\
\frac 12\, x^{\t-\l}<\, gt^2\le\, 2x^{\t-\l}}}\ 
\sum_{\frac{x^{\t/2}}{gt} < k \le \frac{2x^{\t/2}}{gt}}\ \, \sideset{}{^*}\sum_{\chi'\, (\mod gk^2)} \Bigg|
\sum_{\substack{m\, \le\, M\\
(m,t)\, = \, 1}} c_m \chi'(m)\Bigg|^2.
 \end{equation}
Here $\small{\sum}^*$ denotes a sum restricted to primitive characters. By a standard inequality \cite[Chapter 27, (10)]{dav},
 \begin{align*}
\sideset{}{^*}\sum_{\chi'(\mod gk^2)}& \Bigg| \sum_{\substack{m\, \le \, M\\
(m,t) \, = \, 1}} c_m \chi'(m)\Bigg|^2\\[2mm]
&\le \frac{\phi(gk^2)}{gk^2} \sum_{\substack{
a\, =\, 1\\
(a,\, gk^2)\, = \, 1}}^{gk^2} \Bigg|
\sum_{\substack{
m\, =\, 1\\
(m,\, t)\, = \, 1}} c_m e
\left(\frac{am}{gk^2}\right)\Bigg|^2
 \end{align*}
Using \Cref{lem2} for fixed $g$ and $t$, the sum over $k$ on the right-hand side of \eqref{eq2.7} is
 \[\ll (QM)^{\e/3}\left(g\left(\frac{Q^{1/2}}{gt}\right)^3
 + \frac{Q^{1/4}}{(gt)^{1/2}}\, M\right) \|c\|_2^2.\]
(Note that $g \ll x^{\t-\l} \ll M$ here.) For some $G \ge 1$, $T \ge 1$ with $GT^2 \asymp x^{\t-\l}$, we have
 \begin{align*}
T(\l) &\ll (\log x)^2(QM)^{\e/3} \sum_{G \le g < 2G}\ \sum_{T \le t < 2T} \left\{g\left(\frac{Q^{1/2}}{gt}\right)^3 + \frac{Q^{1/4}}{(gt)^{1/2}}\, M\right\} \|c\|_2^2\\[2mm]
&\ll x^\e (Q^{3/2} G^{-1} T^{-2} + Q^{1/4}MG^{1/2} T^{1/2}) \|c\|_2^2\\[2mm]
&\ll x^\e (Q^{1/2} x^\l + Q^{3/4} Mx^{-\l/2})\|c\|_2^2.
 \end{align*}
This completes the proof of \Cref{lem3}.
 \end{proof}

 \begin{lem}\label{lem4}
Let $1 \le x^\l \le Q \ll x^{1/2-\e}$, $x^\l \ge Q^{1/2} x^{\e/6}$. Let $H$ and $K$ satisfy
 \[Qx^{-\l} \ll K \ll H \ll x^{3/5}, \ HK \ll x.\]
Let $a_n$ $(K < n \le 2K)$ and $b_m$ $(H < m \le 2H)$ be complex numbers, $a_n \ll x^\d$, $b_m\ll x^\d$. Let
 \begin{gather*}
K(s, \chi) = \sum_{K\, <\, n\, \le\, 2K} a_n \chi(n) n^{-s},\\[2mm]
H(s, \chi) = \sum_{H\, <\, m\, \le\, 2H} b_m \chi(m) m^{-s},\\[2mm]
S = \sum_{Q\, <\, q^2\, \le\, 2Q} \ \sum_{\substack{
\chi\ (\mod q^2)\\
x^\l\, <\, C(\chi)\, \le\, 2x^\l}} \left|H\left(\frac 12 + it, \chi\right) K\left(\frac 12 + it, \chi\right)\right|.
 \end{gather*}
Then
 \[S \ll x^{1/2-\e/20} Q^{1/2}.\]
 \end{lem}
 
 \begin{proof}
We apply the Cauchy-Schwarz inequality to $S$, followed by applications of \Cref{lem3} to each of the two sums over $q$, $\chi$. The conditions
 \[H \gg x^{\t - \l} \  , \ K \gg x^{\t - \l}\]
are fulfilled since
 \[H \ge K \gg Qx^{-\l}.\]
Since $\sum\limits_m |b_m m^{-\frac 12 - it}|^2 \ll x^{2\d}$ and similarly for $\sum\limits_n|a_n n^{-\frac 12 - it}|^2$, we have
 \begin{align*}
S &\ll x^{3\d} (Q^{1/4}x^{\l/2} + Q^{3/8} K^{1/2} x^{-\l/4})(Q^{1/4} x^{\l/2} + Q^{3/8} H^{1/2} x^{-\l/4})\\[2mm]
&\ll x^{3\d} (Q^{1/2} x^\l + Q^{3/4} x^{1/2-\l/2} + Q^{5/8} x^{\l/4} H^{1/2})\\[2mm]
&\ll x^{3\d} (Q^{3/2} + Q^{3/4} x^{1/2-\l/2} + Q^{7/8} x^{3/10}).
 \end{align*}
Each of these three terms is $\ll x^{1/2 - \e/20}Q^{1/2}$:
 \begin{align*}
&Q^{3/2} x^{3\d} (x^{1/2-\e/20} Q^{1/2})^{-1} \ll Qx^{\e/20 + 3\d - 1/2} \ll 1;\\[2mm]
&Q^{3/4} x^{1/2 - \l/2 + 3\d} (x^{1/2-\e/4} Q^{1/2})^{-1} \ll Q^{1/4}x^{-\l/2 + 3\d + \e/20} \ll 1;\\[2mm]
&Q^{7/8} x^{3/10 + 3\d} (Q^{1/2} x^{1/2-\e/20})^{-1} \ll Q^{3/8} x^{-1/5 + \e/20+ 3\d} \ll 1. 
 \end{align*}
This completes the proof of \Cref{lem4}.
 \end{proof}

\section{Proof of \Cref{thm1}}

It is convenient to write $S(Q) = \{q^2 : Q < q^2 \le 2Q\}$.

 \begin{lem}\label{lem5}
Let $0 < \g < 1$. There is a subset $F(Q)$ of $S(Q)$ with
 \[\#\, F(Q) \ll Q^{1/2-\b},\]
such that for $q^2 \in S(Q)\backslash F(Q)$, $\chi$ a nonprincipal character $(\mod q^2)$ and ${\rm Re}\, s = 1/2$, we have
 \[\sum_{n\, \le\, N} \chi(n)n^{-s} \ll |s|\,
 N^{\frac 12-\b}\quad (N \ge q^\g).\]
Here $\b = \b(\g) > 0$. The implied constants depend on $\g$.
 \end{lem}
 
 \begin{proof}
This is a special case of \cite[Lemma 6]{rcb}.
 \end{proof}
 
We shall refer to $F(Q)$ in the remaining lemmas. The following lemma is a variant of \cite[Proposition 1]{rcb}.

 \begin{lem}\label{lem6}
Let $M_1, \ldots, M_{15}$ be numbers with $M_1 \ge \cdots \ge M_{15} \ge 1$, and suppose that $\{1, \ldots, 15\}$ has a partition into subsets $A$, $B$ such that 
 \[\prod_{i\, \in\, A} M_i \ll x^{1/2-3\e/4}, \
 \prod_{i\, \in\, B} M_i \ll x^{1/2 - 3\e/4}.\]
Let $a_i(m)$ $(M_i/2 < m \le M_i, 1 \le i\le 15)$ be complex sequences with
 \[|a_i(m)| \le \log m \quad (1 \le i \le 15, M_i/2
 < m \le M_i).\]
Suppose that, whenever $M_i > x^{1/8}$, $a_i(m)$ is $1\ (M_i/2 < m \le M_i)$ or $\log m$ $(M_i/2 < m \le M_i)$. Let
 \begin{gather*}
M_i(s, \chi) = \sum_{M_i/2\, < \, m \, \le\, M_i} a_i(m) \chi(m) m^{-s},\\[2mm]
L = x/(M_1 \ldots M_{15})\, , \, B_1(s, \chi) = \sum_{Lx^{-\e}\, <\, n\, \le\, L} \chi(n)n^{-s}.
 \end{gather*}
Then for ${\rm Re}\, s = 1/2$ and $Q \ll x^{1/2-\e}$,
 \[S : = \sum_{q\,\in\, S(Q)\backslash F(Q)}\
 \sum_{\substack{\chi\, (\mod q)\\
 \chi\, \ne\, \chi_0}} |B_1(s,\chi) M_1(s, \chi) 
 \ldots M_{15} (s, \chi)| \ll |s|^3 Q^{1/2} x^{1/2-3\d}.\]
 \end{lem}
 
 \begin{proof}
It suffices to show for $0 \le \l \le \t$ that 
 \begin{equation}\label{eq3.1}
S(\l) \ll |s|^3 Q^{1/2} x^{1/2-4\d},
 \end{equation}
where $S(\l)$ is the subsum of $S$ defined by the additional condition
 \[x^\l < C(\l) \le 2x^\l.\]

Arguing exactly as in the proof of \cite[Lemma 10]{rcb}, \eqref{eq3.1} holds unless (writing as usual $Q = x^\t$) we have
 \begin{equation}\label{eq3.2}
\l > (5\t + \e)/6,
 \end{equation}
We now suppose that \eqref{eq3.2} holds. We decompose $B_1(s, \chi)$ into $O(\log x)$ subsums $M_{16}(x,\chi)$ defined by a condition
 \[M_{16}/2 < n \le M_{16},\]
where $Lx^{-\e} \le M_{16} < L$. It suffices to prove the analogue of \eqref{eq3.1} with $B(s, \chi)$ replaced by $M_i(s, \chi)$ and $6\d$ in place of $4\d$.

Rearranging $M_1, \ldots, M_{16}$ as $N_1 \ge \cdots \ge N_{16}$, write $N_i(s, \chi)$ for the corresponding Dirichlet polynomials and
 \[N_i = x^{\b_i}.\]
Then $\b_1 \ge \cdots \ge \b_{16} \ge 0$, $1 - \e \le \b_1 + \cdots + \b_{16} \le 1$.

We can use the argument in the proof of \cite[Lemma 15]{rcb} to complete the present proof whenever $\b_1 + \b_2 > 3/5$. Suppose now that
 \[\b_1 + \b_2 < 3/5.\]

As shown in the proof of \cite[Lemma 15]{rcb}, there is a subset $W$ of $\{1, \ldots, 16\}$ such that
 \[x^{1/2} \ll H : = \prod_{j\, \in\, W} 2M_j
 \ll x^{3/5}.\]
Let $K := \prod\limits_{\substack{j\, \le\, 16\\
j\, \not\in\, W}} 2M_j$. We see that
 \[x^{2/5-\e} \ll K \ll H, \quad HK \ll x.\]
Let
 \[H(s, \chi) = \prod_{j\, \in\, W} M_j(s, \chi),
 \ K(s,\chi) = \prod_{\substack{1\, \le\, j\, \le\, 16\\
 j\, \not\in\, W}} M_j(s, \chi).\]
We note that
 \[K \gg x^{\t-\l}, \ \text{ since } \ \t - \l < 1/12.\]
Hence we may apply \Cref{lem3} to obtain the desired bound in the form
 \[\sum_{q\, \in\, S(Q)} \ \sum_{\chi\ (\mod q)}
 |H(s, \chi) K(s, \chi)| \ll x^{1/2 - 6\d} Q^{1/2}.\qedhere\]
 \end{proof}

Our final lemma is a variant of \cite[Lemma 18]{rcb}.

 \begin{lem}\label{lem7}
Let $a_i(m)$ $(1 \le i \le 15)$ be nonnegative sequences satisfying the hypotheses of \Cref{lem6}. Let
 \[u_d = \sum_{\substack{d\, =\, m_1\ldots m_{15}\\
 M_i/2\, <\, m_i\, \le\, M_i\ \forall_i}} a_1(m_1)
 \ldots a_{15}(m_{15})\]
for $D_1 < d \le D$, with $D = M_1\ldots M_{15}$, $D_1 = 2^{-15}D$. Let $Q \ll x^{1/2-\e}$. Then for every $A > 0$,
 \[\sum_{q\, \in\, S(Q)\backslash F(Q)} \Bigg|
 \sum_{D_1\, <\, d\, \le\, D} u_d r_0(x, q, a^{(q)},
 d)\Bigg| \ll \frac x{Q^{1/2}(\log x)^A}.\]
 \end{lem}
 
 \begin{proof}
Just as in the proof of \cite[Lemma 18]{rcb}, it suffices to show that
 \[\sum_{q\, \in\, S(Q)\backslash F(Q)} \Bigg|
 \sum_{D_1\, <\, d\, \le\, D} u_d r_4(x, q, a^{(q)},
 d)\Bigg| \ll x^{1-\d} Q^{-1/2}\]
The condition from small $\ell$ in \eqref{eq1.2}, \eqref{eq1.3} to $r_4$ is negligible:
 \begin{align*}
\sum_{D_1\, <\, d\, \le\, D} |u_d| \, |r_4(x^{1-\e}, q, a^{(q)}, d)| &\ll x^{\e/3} \Bigg\{\sum_{\substack{md\, \le\, x^{1-\e}\\
md\, \equiv\, a\, (\mod q)}} + \sum_{d\, \le\, x^{1-2\e/3}} \frac x{qd}\Bigg\}\\[2mm]
&\ll x^{1-\d} Q^{-1};
 \end{align*}
thus it suffices to show that 
 \begin{equation}\label{eq3.3}
\sum_{q\, \in\, S(Q)\backslash F(Q)}\ \sum_{D_1\, <\, d\, \le\, D} u_d(r_4(x, q, a^{(q)}, d) - r_4(x^{1-\e}, q, a^{(q)}, d)) \ll x^{1-\d} Q^{-1/2}. 
 \end{equation}
We now follow the argument in the proof of \cite[Lemma 18]{rcb} to show that \eqref{eq3.3} follows from
 \begin{equation}\label{eq3.4}
\int_{{\rm Re}\, s = 1/2} \sum_{q\, \in\, S(Q)\backslash F(Q)}\ \sum_{\substack{\chi\, (\mod q)\\
\chi\, \ne\, \chi_0}} \Bigg| \sum_{D_1\, <\, d\, \le\, D} u_d \chi(d) d^{-s}\Bigg| |B_1(s,\chi)| \frac{|ds|}{|s|^5} \ll x^{1/2-\d} Q^{-1/2}. 
 \end{equation}
Here $B_1(s, \chi)$ is the Dirichlet polynomial in \Cref{lem6}. At this point we see that \eqref{eq3.4} follows from \Cref{lem6}.
 \end{proof}

 \begin{proof}[Proof of \Cref{thm1}]
Just as in \cite{rcb}, we reduce this to showing that
 \begin{align}
\sum_{q\, \in\, S(Q)\backslash F(Q)} \Bigg| & \sum_{m,\, n\, \le\, Qx^{\e/4}} \Lambda(m) \mu(n) r_0(x, q, a^{(q)}, mn)\Bigg|\label{eq3.5}\\[2mm]
&\ll xQ^{-1/2}(\log x)^{-A}.\notag 
 \end{align}
for every $A > 0$. We use Heath-Brown's decomposition of $\Lambda(m)$, and a slight variant of this decomposition for $\mu(n)$, to show that \eqref{eq3.5} follows from \Cref{lem7}; full details are given on page 158 of \cite{rcb}. This completes the proof of \Cref{thm1}.
 \end{proof}


\begin{thebibliography}{www}

\bibitem{bz} S.~Baier and L.~Zhao, \textit{An improvement for the large sieve for square moduli}, J.~Number Theory \textbf{128} (2008), 154--174.

\bibitem{rcb} R.~C.~Baker, \textit{Primes in arithmetic progressions to spaced moduli}, Acta Arith. \textbf{153} (2012), 133--159.

\bibitem{dav} H.~Davenport, \textit{Multiplicative Number Theory}, 3rd ed., Springer, 2000.

\bibitem{mont} H.~L.~Montgomery, \textit{Topics in Multiplicative Number Theory}, Springer, 1971.


 \end{thebibliography}
 \end{document}